\documentclass[12pt]{amsart}
\usepackage{amsfonts}
\usepackage{amscd}
\usepackage{amssymb}
\usepackage{graphics}

\usepackage{amsmath}

\usepackage{hyperref}
\hypersetup{colorlinks,citecolor=blue}

\newtheorem{theorem}{Theorem}

\newtheorem{conjecture}[theorem]{Conjecture}

\newtheorem{problem}[theorem]{Problem}

\begin{document}
\author{Kate Juschenko} 
\address{Kate Juschenko, Department of Mathematics,
Northwestern University, 2033 Sheridan Road, Evanston, IL 60208, USA}
\email{kate.juschenko@gmail.com}

\title{Liouville property of strongly transitive actions}
\begin{abstract}
Liouville property of actions of discrete groups can be reformulated in terms of existence of co-F$\o$lner sets. Since every action of amenable group is Liouville, the property can be used for proving  non-amenability. There are many groups that are defined by strongly transitive actions. In some cases amenability of such groups is an open problem. We define $n$-Liouville property of  action to be Liouville property of point-wise action of the group on the sets of cardinality $n$.  We reformulate $n$-Liouville property in terms of additive combinatorics and prove it for $n=1, 2$. The case $n\geq 3$ remains open.

\end{abstract}

\maketitle

\section{Introduction}

Let $G$ be a countable discrete group acting on a set $X$. A probability measure $\mu$ on $G$ is non-degenerate if $supp(\mu)$
generate  $G$ as semigroup. Define transition probabilities from $x$ to $y$ in $X$ induce by the measure $\mu$ by
$$p_{\mu}(x,y)=\sum\limits_{g\in G, gx=y}\mu(g).$$


A function $f:X\to\mathbb{R}$ is ${\mu}$-harmonic if $$f(x)=\sum_{y\in X}f(y)p_{\mu}(x,y).$$

The action of $G$ on $X$ is  called $\mu$-Liouville if all bounded $\mu$-harmonic
functions are constant.  In particular, this definition implies that $\mu$-Liouville actions are transitive. We call an action Liouville
if there is a non-degenerate measure $\mu$ on $G$ which makes it $\mu$-Liouville.

One of the main motivations to study Liouville actions is their relation to amenability.  A  renowned result of Kaimanovich and Vershik, \cite{KaimVersh}, states that the group is not amenable if and only if the action of the group on itself is not Liouville. However, if a transitive action of $G$ on a set is not Liouville then the action of $G$ on itself is not Liouville either. This follows from a simple fact that if $f:X\rightarrow \mathbb{R}$ is $\mu$-Liouville then for every $x\in X$ the function $f_x:G\rightarrow \mathbb{R}$ defined by $f_x=f(gx)$ is $\mu$-Liouville. Thus, having a bounded non-constant $\mu$-harmonic function on $X$ one obtains  a bounded non-constant $\mu$-harmonic on $G$.

This approach to non-amenability have been recently used by
Kaimanovich, \cite{Kaimanovich}, as a suggested approach to show non-amenability of Thompson group F.  In particular, Kaimanovich
showed that  for every finitely supported non-degenerate measure $\mu$ on Thompson
group F the action of the group on dyadic rationals is not $\mu$-Liouville. In \cite{JZ}, Zheng and the author showed that this action is Liouville. An action of $G$ on $X$ is called strongly transitive if for every finite subsets $E$ and $F$ of $X$ of the same cardinality there is $g$ in $G$ such that $g(E)=F$.  The main property of the Thompson group F that was used in \cite{JZ} is strong transitivity of order preserving action on a totally ordered set and the  existence of an element in the group with an infinite orbit. \\

This paper aims to study Liouville property of actions on the susbsets of a fixed cardinality. An action of a group $G$ on a set $X$ is called $n$-Liouville if there is a non-degenerate measure $\mu$ on the subsets of $X$ of cardinality $n$ such that every bounded $\mu$-harmonic function is constant. One of the main ingredients will be the following theorem/definition of Liouville action.\\

Let $\mu$ be a finitely additive measure on $G$ and $x\in X$, denote by $(m\cdot x)$ the finitely additive measure on $X$ given by
 $(m\cdot x) (A)=m(\{g\in G:gx\in A\})$.

\begin{theorem}\label{thm1}
Let $G$ act transitively on a set $X$, then the following are equivalent:

\begin{enumerate}
\item the action is Liouville, i.e., there exists a non-degenerate symmetric measure $\mu$ on $G$ such that $X$ is $\mu$-Liouville;\\ \label{c1}
\item there exists a finitely additive measure $m$ on $G$ such that $m\cdot x=m\cdot y$ for any $x, y$ on $X$.\\ \label{c2}
\item \label{cond4} for every $\varepsilon>0$ and for every finite $F\subseteq X$ there exists a finite set $E\subset G$ such that we have
$$|Ex\Delta Ey|\leq \varepsilon |E| \text{ for every } x,y\in F. $$
\end{enumerate}

\end{theorem}

One should compare the last condition with the definition of amenable action, i.e., the existence of F\o lner sets: \\

{\it For every $\varepsilon>0$ and for every finite $E\subseteq G$ there exists a finite set $F\subset X$ such that we have
$$|gF\Delta hF|\leq \varepsilon |F| \text{ for every } g,h\in E. $$}

We use the condition (\ref{cond4}) to show that the action of Thompson group F on the dyadic rationals is $1$- and $2$-Liouville. However, the case of $n$-Liouville property for $n\geq 3$ is not currently feasible for us. This question can be reformulated in terms of additive combinatorics. It is tempting to conjecture that Thompson group F does not have $3$-Liouville property. A negative answer to the following question implies that Thompson F does not have $3$-Liouville action, and thus  is not amenable: \\

{\it Is it true that for every $\varepsilon>0$ there is a set $V\subset \mathbb{N}\times\mathbb{N}\times\mathbb{N}$ such that

\begin{gather*}\label{con1}|\{(p_1(x),p_2(x)):x\in V\}\cap \{(p_2(x),p_3(x)):x\in V\}\cap\\ \{(p_1(x)+p_2(x),p_3(x)):x\in V\}\cap \{(p_1(x),p_2(x)+p_3(x)):x\in V\}| \geq (1-\varepsilon) |V|\end{gather*}

where $p_i$ is a projection on the corresponding coordinate. Here we assume possibility that $p_i(x)=p_i(y)$ for distinct $x$ and $y$, and consider $p_i(x)$ as a set with multiplicity.}\\

The question above is open. It is also a subquestion of the following question, which in its turn equivalent to $3$-Liouville property. \\

{\it For every $n\geq 3$ and for every  $\varepsilon>0$ there is a finite set $V$ in $\mathbb{N}\times\mathbb{N}\times\ldots \times\mathbb{N}$ (product taken n times), such that 
\begin{gather*}|\bigcap_{1\leq k\leq n; 1\leq i\leq k; k+1\leq m\leq n}\{(p_i(x)+p_{i+1}(x)+\ldots+p_{k}(x),p_{k+1}(x)+\ldots+p_{m}(x)):x\in V\}|\\ \geq (1-\varepsilon) |V|?
\end{gather*}}

Moreover, the first question (after proper reordering of the set $W$) implies that for every $\varepsilon>0$ there must exist a sequence $a_1,\ldots, a_n$ such that
\begin{align*}|\{(a_i,a_{i+1}):1\leq i\leq n-1\}&\cap \{(a_j+a_{j+1},a_{j+2}):1\leq j\leq n-2\}\\&\cap   \{(a_k,a_{k+1}+a_{k+2}):1\leq k\leq n-2\}|\\&\geq (1-\epsilon)n.
\end{align*}
However, we are unable to build such a sequence or prove that it does not exist.\\ 

 A reformulation of $n$-Liouville property is very similar to the one of $3$-Liouville. Instead of intersecting  elements of  $\mathbb{N}\times\mathbb{N}$ one considers vectors of dimension $n-1$. We will specify it in the later section. \\
 
{\bf Acknowledgements:} We are grateful to Tianyi Zheng and Terrence Tao for numerous discussions on Liouville property and existence of the sets. The Theorem \ref{thm1} was partially suggested by Vadim Kaimanovich at workshop on amenability at AIM.  The Liouville property is also suggested to be called lamenable by L. Bartholdi.

\section{F\o lner type sets and Liouville actions}

Let $\mu$ be a finitely additive measure on $G$ and $x\in X$, denote by $(m\cdot x)$ the finitely additive measure on $X$ given by
 $(m\cdot x) (A)=m(\{g\in G:gx\in A\})$.

\begin{theorem}\label{thm1}
Let $G$ act transitively on a set $X$, then the following are equivalent:

\begin{enumerate}
\item the action is Liouville, i.e., there exists a non-degenerate symmetric measure $\mu$ on $G$ such that $X$ is $\mu$-Liouville;\\ \label{c1}
\item there exists a finitely additive measure $m$ on $G$ such that $m\cdot x=m\cdot y$ for any $x, y$ on $X$.\\ \label{c2}
\item \label{cond4} for every $\varepsilon>0$ and for every finite $F\subseteq X$ there exists a finite set $E\subset G$ such that we have
$$|Ex\Delta Ey|\leq \varepsilon |E| \text{ for every } x,y\in F. $$
\end{enumerate}

\end{theorem}
\begin{proof}

Condition (\ref{c1}) implies  (\ref{c2}): if $\mu$ be a measure such that the action on $X$ is $\mu$-Liouville, then convolution powers of approximate a mean on G such
that (\ref{c2}) is satisfied. To show the converse, given such a mean $m$ on G, let $v_n$ be a sequence of finitely supported
probability measures that approximate $m$, then one can build a measure using the Kaimanovich-Vershik method (or as in \cite{JZ}) which is Liouville. 

The equivalence of (\ref{cond4}) and (\ref{c2}) follows from the same considerations as in the proof of equivalence of  existence of F\o lner sets and approximately invariant $l_1$-functions (Reiter's condition), see \cite{Kbook} for example.
\end{proof}

\section{$1$- and $2$-Liouville actions}
The $1$-Liouville property was shown in \cite{JZ}. In this section we will consider $2-$Liouville property. The case of $1$-Liouville also follows from our construction.

An action of $G$ on $X$ is called strongly transitive if for every finite subsets $E$ and $F$ of $X$ of the same cardinality there is $g$ in $G$ such that $g(E)=F$. Denote by $\mathcal{P}_n(X)$ the set of all subsets of $X$ of cardinality $n$. 

We will consider the group $F_\mathbb{R}$ of all piece-wise linear homeomorphisms of $\mathbb{R}$ with slopes as powers of $2$ and breaking points of the first derivative in dyadic rationals. Thompson group $F$ is the subgroup of $F_\mathbb{R}$ made of the homeomorphisms that are identity outside of the interval $[0,1]$. It contains many other copies of Thompson group $F$, for example subgroups $G_n$ that are identity outside of $[-2^n,2^n]$. The group $F_\mathbb{R}$ itself is an abelian extension of the union of $G_n$. Thus, Thompson group $F$ is amenable if and only if $F_\mathbb{R}$ is amenable. The consideration of $F_\mathbb{R}$ allows scaling of dyadic rationals to natural numbers.

\begin{theorem}\label{thm2} The action of Thompson group F and the group $F_\mathbb{R}$ on dyadic rationals is $2$-Liouville.
\end{theorem}
\begin{proof}Denote by  $X$ the set of dyadic rationals for the case of $F_\mathbb{R}$ and dyadic rationals intersected with $[0,1]$ for Thompson group F.
Let us firstly consider the group $F_\mathbb{R}$ and a set $V$ of all subsets of cardinality $2$.   For every $\varepsilon>0$ we have to find a finite subset $E\subset V$ such that 
\begin{gather}\label{cc1}|Ex\Delta Ey|\leq \varepsilon |E| \text{ for every }x,y\in V.\end{gather}
By multiplying the set $E$ on the right by an element of the form $f(x)=2^i x$ with $i$ large enough, we can assume that all elements of $V$ are the subsets of natural numbers.
For $x=\{x_1,x_2\}\in V$ define $\hat{x}=|x_2-x_1|$, and for a subset $Q$ of $\mathcal{P}_2(X)$ denote $\widehat{Q}=\{\hat{x}:x\in Q\}$ and consider this set with multiplicities. In particular, the condition (\ref{cc1}) obviously implies 
\begin{gather}\label{cc2}|\widehat{Ex}\Delta \widehat{Ey}|\leq \varepsilon |E| \text{ for every }x,y\in V.\end{gather}

We claim that the existence of sets that satisfy the condition (\ref{cc2}) implies the existence of the sets that satisfy the condition (\ref{cc1}). Fix $\varepsilon>0$ and let $E$ be a set that satisfy (\ref{cc2}). Obviously, $Ex$ and $Ey$ might not even intersect. However, we can move the sets $Ex$ and $Ey$ around the way that the intersection will be large. To do that define $g(x)=x+1$. Then the set $E'=\{g,g^2,\ldots,g^n\}\cdot E$ will satisfy 
$$|E'x\Delta E'y|\leq \varepsilon' |E'| \text{ for every }x,y\in V,$$
where $\varepsilon'$ tends to $0$, when $\varepsilon$ tends to $0$. Indeed, it is easy to see that if $x$ and $y$ are two elements of $\mathcal{P}_2(X)$ with $\hat{x}=\hat{y}$, then $$a_n=|\{g,g^2,\ldots,g^n\}x\cap \{g,g^2,\ldots,g^n\}y|/n \rightarrow 1$$ when $n$ goes to infinity. However,
$$|E'x\cap E'y|\geq (1-\varepsilon) a_n  n |E|=(1-\varepsilon) a_n |E'| \text{ for every }x,y\in V,$$
which implies the claim.

In order to show the existence of sets that satisfy the condition (\ref{cc2})
let us consider one of the easiest and most illustrative cases from which the most general case will be clear. Consider the case where $V$ is the set of all subsets of cardinality $2$  which are supported on $4$ points of natural numbers $n_1$, $n_2$, $n_3$ and $n_4$. We claim that if for every $\varepsilon>0$ there exists a set finite set $W\subset \mathbb{N}\times\mathbb{N}\times\mathbb{N}$, considered with multiplicities, such that

\begin{gather*}\label{ineq}|p_1(W)\cap p_2(W)\cap p_3(W)\cap \{p_1(w)+p_2(w):w\in W\}\cap  \{p_2(w)+p_3(w):w\in W\}\cap\\
\cap \{p_1(w)+p_2(w)+p_3(w):w\in W\}|\geq (1-\varepsilon) |W|,
\end{gather*}
then there is a finite subset $E$ of the group that satisfies the condition (\ref{cc2})
Here $p_i$ is the projection of $W$ onto the $i$-th coordinate (the sets $p_i(W)$ are considered with multiplicities). Indeed, fix natural numbers $r_1$, $r_2$, $r_3$, and let $A$ be a set of natural numbers which is almost invariant by multiplication by $r_i$, $r_i+r_j$ and $r_1+r_2+r_3$  for all $i\neq j$, then the set $W=\{(r_1a,r_2a,r_3a):a \in A\}$ satisfies the condition.

Now the set $E'$ is constructed as follows. To each element $(w_1,w_2,w_3)$ in $W$ we associate an element of the group that sends $n_1$ to $0$, $n_2$ to $w_1$, $n_3$ to $w_1+w_2$ and $n_4$ to $w_1+w_2+w_3$. This can be arranged because the group acts strongly transitively on $X$. The set $E'$ is chosen to be the set of all elements in the group associated to the elements of $W$. It is clear that the intersection of projections above guarantees the large intersection of the sets $E'x$ and $E'y$. For example, if $x=\{0,1\}$ and $y=\{1,2\}$, then the fact that the intersection of $p_1(W)$ and $p_2(W)$ is large guarantees that the intersection of $E'x$ and $E'y$ is large.

The case when the elements of $V$ are supported on larger sets, say on a set of cardinality $n$, follows from existence for every $\varepsilon$ of a set $W\subset \mathbb{N}\times\ldots\times\mathbb{N}$ (where the product is taken $n-1$ times) which satisfy
\begin{gather*}\label{ineq}|p_1(W)\cap\ldots\cap p_{n-1}(W)\cap\bigcap\limits_{1\leq i\leq{n-2}} \{p_i(w)+p_{i+1}(w):w\in W\}\cap \\\bigcap\limits_{1\leq i\leq{n-3}} \{p_i(w)+p_{i+1}(w)+p_{i+2}(w):w\in W\}\cap \ldots\\\ldots \cap \{p_1(w)+p_2(w)+\ldots p_{n-1}(w):w\in W\}|\geq (1-\varepsilon) |W|,
\end{gather*}
Such sets exist and can be chosen of the form $W=\{(r_1a,r_2a,\ldots,r_{n-1}a):a \in A\}$, where $A$ is a set of natural numbers invariant under multiplication of corresponding sums of $r_i$.
\end{proof}

The theorem below can be proved in a similar way.

\begin{theorem} Let a discrete group $G$ act strongly transitively on a totally ordered set $X$ in an order preserving way. Assume that there is a sequence of elements $g_n$ in $G$ which admits arbitrarily large orbits. Then the action is $Liouville$.
\end{theorem}

\section{n-Liouville actions and open problems}
Let us first consider the case $n=3$ and state the reformulation of $3$-Liouville property. The following theorem is a straightforward adaptation of the first part of the proof of Theorem \ref{thm2}.
\begin{theorem}\label{tl} The action of Thompson group $F$ or the group $F_{\mathbb{R}}$ on dyadic rationals is $3$-Liouville if and only if for every $n$ and $\varepsilon>0$ there exists a set $W\subset \mathbb{N}\times\ldots \times \mathbb{N}$ (n-fold product) such that 

\begin{gather*}|\bigcap_{1\leq k\leq n; 1\leq i\leq k; k+1\leq m\leq n}\{(p_i(x)+p_{i+1}(x)+\ldots+p_{k}(x),p_{k+1}(x)+\ldots+p_{m}(x)):x\in W\}|\\ \geq (1-\varepsilon) |W|.
\end{gather*}

\end{theorem}

The general case of $n$-Liouville property is a modification of the intersection above to the intersection of $n-1$ tuples of natural numbers. One of the approaches to prove $n$-Liouville property is the following conjecture.

\begin{conjecture} There exists an amenable group that act strongly transitively on a totally ordered set $X$ in an order preserving way.
\end{conjecture}

If the conjecture is true then by \cite{JZ} we can deduce that the action of any group (in particular, Thompson group $F$) that satisfy the theorem is $n$-Liouville for any $n$. Note that there are examples of amenable groups that act strongly transitively on sets and admit elements of infinite orbits, for example those coming from topological full groups \cite{JM}, \cite{JNS}.

We are tempted to conjecture that the sets in the Theorem \ref{tl}  do not exist. This would imply non-amenability of Thompson group F. Even the following seeming easy subproblem is currently out of our reach.
\begin{problem}Is it true that for every  $\varepsilon>0$ there is a finite subset $W$ of $\mathbb{N}\times\mathbb{N} \times\mathbb{N}$, such that
\begin{gather*}\label{con1}|\{(p_1(x),p_2(x)):x\in V\}\cap \{(p_2(x),p_3(x)):x\in V\}\cap\\ \{(p_1(x)+p_2(x),p_3(x)):x\in V\}\cap \{(p_1(x),p_2(x)+p_3(x)):x\in V\}| \geq (1-\varepsilon) |V|\end{gather*}
\end{problem}
The problem above is equivalent to the existence of sets from Theorem \ref{thm1} (\ref{cond4}) when $V$ is chosen to be the set of all subsets of cardinality $3$ supported on $4$ points. 

After  proper reordering of the set $W$ the problem above implies the following problem 

\begin{problem} For every $\varepsilon>0$ there must exist a sequence $a_1,\ldots, a_n$ such that
\begin{align*}|\{(a_i,a_{i+1}):1\leq i\leq n-1\}&\cap \{(a_j+a_{j+1},a_{j+2}):1\leq j\leq n-2\}\\&\cap   \{(a_k,a_{k+1}+a_{k+2}):1\leq k\leq n-2\}|\\&\geq (1-\epsilon)n.
\end{align*}
\end{problem}

The following problem should likely have a negative solution.
\begin{problem}
Assume that a group $G$ acts faithfully on a set $X$ and assume that this action is $n$-Liouville for all $n\in \mathbb{N}$. Is $G$ amenable?
\end{problem}

\end{document}